\documentclass{amsart}

\usepackage[utf8]{inputenc}
\usepackage{amsmath,amssymb,amsthm,arydshln,mathabx,stmaryrd,mathrsfs}
\usepackage[dvips]{graphicx}
\usepackage[american]{babel}
\usepackage[pdfencoding=auto,psdextra]{hyperref}
\DeclareMathAlphabet{\mathbbo}{U}{bbold}{m}{n}

\def\Z{\mathbb{Z}}

\newcommand{\N}{\ensuremath{\mathbb{N}}}

\newcommand{\Zd}{\ensuremath{\mathbb{Z}^d}}

\newcommand{\R}{\ensuremath{\mathbb{R}}}

\newcommand{\C}{\ensuremath{\mathbb{C}}}

\renewcommand{\P}{\ensuremath{\mathbb{P}}}

\renewcommand{\epsilon}{\varepsilon}

\renewcommand{\Re}{\text{Re }}

\newcommand{\miniop}[3]{%
\renewcommand{\arraystretch}{0.6}
\begin{array}{c}
{\scriptstyle #1}\\
#2\\
{\scriptstyle #3}
\end{array}
\renewcommand{\arraystretch}{1}}

\DeclareMathOperator{\Log}{Log}
\DeclareMathOperator{\Dir}{Dir}

\newcommand{\1}{\mathbbo{1}}

\newtheorem{theo}{Theorem}
\newtheorem{lemma}{Lemma}
\newtheorem{coro}{Corollary}

\makeatletter
\def\legendre@dash#1#2{\hb@xt@#1{%
\kern-#2\p@
\cleaders\hbox{\kern.5\p@
\vrule\@height.2\p@\@depth.2\p@\@width\p@
\kern.5\p@}\hfil
\kern-#2\p@
}}
\def\@legendre#1#2#3#4#5{\mathopen{}\left(
\sbox\z@{$\genfrac{}{}{0pt}{#1}{#3#4}{#3#5}$}%
\dimen@=\wd\z@
\kern-\p@\vcenter{\box0}\kern-\dimen@\vcenter{\legendre@dash\dimen@{#2}}\kern-\p@
\right)\mathclose{}}
\newcommand\legendre[2]{\mathchoice
{\@legendre{0}{1}{}{#1}{#2}}
{\@legendre{1}{.5}{\vphantom{1}}{#1}{#2}}
{\@legendre{2}{0}{\vphantom{1}}{#1}{#2}}
{\@legendre{3}{0}{\vphantom{1}}{#1}{#2}}
}
\def\dlegendre{\@legendre{0}{1}{}}
\def\tlegendre{\@legendre{1}{0.5}{\vphantom{1}}}
\makeatother

\author{Olivier \textsc{Garet}}
\title{How often is $x\mapsto x^3$ one-to-one in $\Z/n\Z$~?}

\newcounter{thmletter}

\newtheorem{theoremletter}[thmletter]{Theorem}
\newtheorem{remark}{Remark}

\begin{document}

\begin{abstract}
We characterize the integers $n$ such that $x\mapsto x^3$ describes a bijection from the set $\Z/n\Z$ to itself and we determine the frequency of these integers. Precisely, denoting by $W$  the set of these integers, we prove that an integer belongs to $W$ if and only if it is square-free with no prime factor that is congruent to 1 modulo 3, 
and that there exists $C>0$ such that
$$|W\cap\{1,\dots,n\}|\sim C\frac{n}{\sqrt{\log n}}.$$
These facts (or equivalent facts) are stated without proof on the OEIS~\cite[A074243]{oeis} website. 
We give the explicit value of $C$, which did not seem to be known. 
Analogous results are also proved for families of integers for which congruence classes for prime factors are imposed. 
The proofs are based on a Tauberian Theorem by Delange. 
\end{abstract}

\maketitle
\section{Introduction}

It is well known that the map $x\mapsto x^3$ realizes a bijection on the real line. It is natural to wonder what the situation is in $\Z/n\Z$. The study of a few examples quickly shows that the integers $n$ that enjoy this property are relatively numerous: with the help of computer tools, we can, for example, draw up a list of them between $1$ and $50$: 1, 2, 3, 5, 6, 10, 11, 15, 17, 22, 23, 29, 30, 33, 34, 41, 46, 47.

This sequence is referenced on the On-Line Encyclopedia of Integer Sequences (OEIS)~\cite[A074243]{oeis} website, 
which lists many remarkable sequences.
Some properties of this sequence are stated there, without reference to proof.
In particular, it is stated that there is a constant $c$ such that the $n$-th term of the sequence is equivalent to $c n\sqrt{\log n}$.

In fact, we will prove the following result, which is equivalent, but whose statement seems more meaningful:


Noting $W$ as the set of integers $n$ such that the application $x\mapsto x^3$ realizes a bijection of $\Z/n\Z$ to itself,  there exists $C>0$ such that
$$|W\cap\{1,\dots,n\}|\sim C\frac{n}{\sqrt{\log n}},$$

which is equivalent to the result announced by the OEIS, with $C=\frac1{c}$. We give a literal value for $C$, which is close to $0.664$.

The natural density of this set is therefore zero, but those numbers are less spaced out than prime numbers. 
This is not surprising. According to the Chinese theorem, if $p$ and $q$ are coprime integers and $a$ is a cube modulo $p$ and modulo $q$,
 it will also be a cube modulo $pq$, and vice versa. So, the fact that $W$ is stable for the product of coprime elements mechanically generates 
 some combinatorics.

This may bring to mind another set of integers, also stable through the product of coprime integers: 
the set of integers $S$ that can be written as the sum of two squares. It is known that 
$$|S\cap\{1,\dots,n\}|\sim b\frac{n}{\sqrt{\log n}},$$
where $b$ is the Landau-Ramanujan constant, which is close to $0.764$~\cite[A064533]{oeis}.
This result, demonstrated by Landau~\cite{zbMATH02640396} and rediscovered by Ramanujan, is frequently mentioned in monographs on analytic number theory (see, for example, Tenenbaum~\cite{zbMATH00989379}).

This remark gives us the method for exploring the problem: to characterize the elements of $W$, we need to determine under which conditions $p^{\alpha}$ is in $W$ when $p$ is a prime number.
This is the subject of section 2, where we will show that $p^{\alpha}\in W$ if and only if $\alpha=1$ and $p$ is not congruent to $1$ modulo 3.


The determination of an equivalent of the number of integers smaller than $x$ whose prime factors belong to a family of congruences modulo $m$ that are coprime with $m$ is an old problem, solved at the beginning of the 20th century by Landau~\cite{zbMATH02636963}. A little over half a century later, Wirsing revisited the problem, in order to link the number of integers less than $x$ whose prime factors belong to a certain family of prime numbers to the regularity of this family~\cite{zbMATH03125407,zbMATH03169573}.
This problem is mentioned in Tenenbaum's monograph~\cite{zbMATH00989379}, where he proposes an exercise (or rather a series of exercises) leading, in the spirit of Wirsing, to the obtention of an explicit equivalent for the number of integers not exceeding $x$ whose prime factors are all of the form $4m+3$.

The question of the existence and calculation of a natural density for integers without a square factor is much better known; 
it can be found as an exercise in many mathematics textbooks (see for example Garet--Kurtzmann~\cite{zbMATH07063635}). 
The result is often attributed to Gegenbauer~\cite{zbMATH02699867}, but it was demonstrated the same year by Ces\`{a}ro~\cite{cesaro3}. 
\footnote{The articles by Gegenbauer and Ces\`{a}ro were published in 1885, however it can be noted that a similar result on pairs of coprime integers had already been demonstrated at the same time by several authors~\cite{zbMATH02716583,zbMATH02703467}.}

Our bibliographic research has not led us to a situation that combines both of these conditions.

It is time to state our main results.
We start with the general theorem:

\begin{theo}\label{montheo}
Let $Q$ be the set of square-free integers. For $m$ a positive natural integer and $A\subset \{1,\dots,m\}$, 
let $I_m(A)$ be the set of positive natural integers for which none of the prime factors is congruent modulo $m$ to an element of $A$.

There exist non-zero positive constants $(c_a)_{1\le a\le m}$ such that the following result holds: for $A\subset\{1,\dots,m\}$, 
provided that $\ell=|\{a\in A; a\wedge m=1\}|<\phi(m)$, where $\phi$ denotes the Euler totient,
then $n\to\infty$ we have
\begin{align}
|I_m(A)\cap \{1,\dots,n\}|\sim \left(\prod_{a\in A}c_a\right) \frac{n}{(\log n)^{\frac{\ell}{\phi(m)}}}
\end{align}
and
\begin{align}
|I_m(A)\cap Q\cap \{1,\dots,n\}|\sim\left(\prod_{a\in A}c_a\right) \prod_{\substack{p\not\in A+m\Z\\ p\text{\emph{ prime}}}} \left(1-\frac1{p^2}\right) \frac{n}{(\log n)^{\frac{\ell}{\phi(m)}}}.
\end{align}
In the degenerate case where $\ell=\phi(m)$, we denote by $B$ the set of prime divisors of $m$ that are not in $A$. The elements in $I_m(A)$ are the integers whose prime factors are in $B$. If $B$ is empty, $I_m(A)$ is reduced to $\{1\}$, otherwise we have
\begin{align}
|I_m(A)\cap \{1,\dots,n\}|\sim \frac{(\log n)^{|B|}}{|B|!}\prod_{p\in B}\frac1{\log p} .
\end{align}
Of course, if $\ell=\phi(m)$, $I_m(A)\cap Q$ is finite, with $|I_m(A)\cap Q|=2^{|B|}$.
\end{theo}

We can easily deduce the desired result:

\begin{theo}\label{theocubes}
Let $W$ be the set of integers $n$ such that $x\mapsto x^3$ realizes a bijection from $\Z/n\Z$ to itself.
There exists a constant $C>0$ such that, at infinity
$$ |\{W\cap\{1,\dots,n\}|\sim C\frac{n}{(\log n)^{1/2}}.$$
\end{theo}

Like Landau~\cite{zbMATH02636963}, we will use the techniques of analytic number theory initiated by Dirichlet, in particular, that of L-functions. However, our proof will be shorter than Landau's, taking advantage of the time that has polished some classical arguments and new results stated since then.
In particular, we will rely on a theorem by Delange, which will give us access to a fairly short proof, 
essentially based on the classical tools that appear in a proof of the arithmetic progression theorem, as can be found, for example, in Serre~\cite{zbMATH03585552}.
This is the subject of section~\ref{III}.

The last section is devoted to the determination of the constant $C$.
It borrows from the aforementioned work by Tenenbaum the idea of factoring the Dirichlet series naturally associated with the problem by a power of the function $L$ associated with a primitive character.

\begin{theo}\label{theconstant}
Under the hypotheses of Theorem~\ref{theocubes}, we have 
\[C=\frac{4}{\pi}3^{-3/4}\left(2\miniop{}{\prod}{\substack{p\equiv 2 \pmod{3}\\ p\text{\emph{ prime}}}} \right)^{1/2}\approx 0.664.\] 
\end{theo}

It can be seen that, for the most part, the proof of Theorem~\ref{theconstant} contains the proof of the existence of $C$, 
borrowing very little material from the proofs of Theorems~\ref{montheo} and~\ref{theocubes}. However, we wanted to present the proof of Theorem~\ref{montheo}, which seems to us to have its own interest.

For example, Theorem~\ref{montheo} allows us to precise a recent result by Brown~\cite{zbMATH07395864}. 
In this article, Brown demonstrates that if $r$ and $m$ are coprime integers, with $1\le r\le m$, then the natural density of $Q\cap I_m(\{r\})$ is zero, which is an immediate consequence of our more precise result:
$$|I_m(\{r\})\cap Q\cap\{1,\dots,n\}|\sim c_r \prod_{\substack{p\ne r\\ p\text{{ prime}}}} \left(1-\frac1{p^2}\right) \frac{n}{(\log n)^{\frac{1}{\phi(m)}}}.$$

\section{Determination of \texorpdfstring{$W$}{W}}\label{II}

We have already noted, thanks to the Chinese Theorem, that \mbox{$N=\prod_{i=1}^{k} p_i^{n_i}$} is in $W$ if and only if for all $i$, $p_i^{n_i}$ is in $W$.

\subsection{The elements of \texorpdfstring{$W$}{W} are square-free}

We first prove that for $p$ prime, if every integer is a cube modulo $p^{n}$, then necessarily $n=1$.
Indeed, if $n\ge 2$ and every integer is a cube modulo $p^n$, then every integer is a cube modulo $p^2$, in particular $p$ 
is a cube modulo $p^2$. However, for $p$ prime, $p$ is never a power modulo $p^2$.
Indeed, if $p$ were congruent to $a^{\alpha}$ modulo $p^2$ with $\alpha\ge 2$, we would have an integer $k$ such that

$$p=a^{\alpha}+kp^2.$$

Thus $a^{\alpha}$ would be congruent to $0$ modulo $p$, and therefore so would $a$; for a certain integer $b$ we could write:
$$p=(pb)^{\alpha}+kp^2.$$

This would imply that $p^2$ divides $p$, which is absurd.

\subsection{\'Study of prime factors}

All integers are cubes modulo $2$ since $0$ and $1$ are cubes. Similarly, all integers are cubes modulo $3$ since $0,1,-1$ are cubes. We now assume that $p$ is a prime number with $p\ge 5$. 

In $\Z/p\Z$, the only solution to the equation $x^3=0$ is $x=0$,
so $x\mapsto x^3$ is surjective from $\Z/p\Z$ to $\Z/p\Z$ if and only if
$x\mapsto x^3$ is surjective from $(\Z/p\Z)^{\times}$ to ${\Z/p\Z}^{\times}$.
However, $\phi:x\mapsto x^3$ being a group morphism from $(\Z/p\Z)^{\times}$ to itself, it is surjective if and only if it is injective. However

$$\ker \phi=\{x\in(\Z/p\Z)^{\times}; x^3=1\}=\{x\in (\Z/p\Z)^{\times}; (x-1)(x^2+x+1)=0\}.$$

Since $1$ is not a root of the polynomial $x^2+x+1$, $\phi$ is injective if and only if the equation $x^2+x+1=0$ has no roots in $\Z/p\Z$.

Since the discriminant of the equation $x^2+x=1=0$ is $\Delta=1-4=-3$, the morphism is injective if and only if $-3$ is not a square modulo $p$, in other words, using the Legendre symbol:
\begin{align*}
p\in W&\iff \legendre{-3}{p}=-1 \iff \legendre{3}{p}\legendre{-1}{p}=-1\\
&\iff \legendre{3}{p}\legendre{p}{3}\legendre{-1}{p}=-\legendre{p}{3}\\
&\iff (-1)^{\frac{p-1}2 \frac{3-1}{2}} \legendre{-1}{p}=-\legendre{p}{3} \text{ with the quadratic reciprocity law}\\
&\iff 1=-\legendre{p}{3}\\
&\iff p \equiv 2 \pmod{3}\text{ because the squares modulo 3 are $0$ and $1$} 
\end{align*}

Finally, $W$ is composed by the square-free numbers  which no prime divisor is congruent to $1$ modulo~3.

\section{\'Study of the frequency}\label{III}
Recall that a function $f:\N\to \C$ is said to be multiplicative if the identity $f(pq)=f(p)f(q)$ holds whenever $p$ and $q$ are coprime integers.
If the identity holds without restriction on $p$ and $q$, we say that $p$ is strictly multiplicative.

As a consequence of the Chinese Theorem, the indicator function of $W$ is a multiplicative function. 
The same is true of the indicator of the set $V$ of integers whose prime factors are not congruent to $1$ modulo~$3$

This leads to the introduction of the Dirichlet series associated with these sets:

$$L(s)=\sum_{n=1}^{+\infty} \frac{\1_{V}(n)}{n^s}\text{ and }\tilde{L}(s)=\sum_{n=1}^{+\infty} \frac{\1_{W}(n)}{n^s}.$$

To study the distribution of the sets $V$ and $W$, we will use a theorem by Delange, 
which is particularly suited to this kind of problem. 
Delange systematized a method for studying the distribution of sets of integers~\cite{zbMATH03120925}, 
which he applied to many sets of numbers with arithmetic properties . 
His theorem is in fact a reformulation of a generalization he obtained from Ikehara's Tauberian Theorem ~\cite{zbMATH03090449}.

\begin{theoremletter}
Let $A$ be a set of integers, $\rho$ a real number that is not a negative integer or zero. Suppose that for $\Re s>1$

$$\sum_{n \in A} \frac{1}{n^{s}}=(s-1)^{-\rho} a(s)+b(s),$$

with $a$ and $b$ holomorphic for $\Re s \geq 1$ and $a(1) \neq 0$

Then, for $x$ tending to $+\infty$ we have

$$
|A\cap [0,x]| \sim \frac{a(1)}{\Gamma(\rho)} x(\log x)^{\rho-1} .
$$
\end{theoremletter}
Here $(s-1)^{-\rho}$ is to be understood as $\exp(-\rho\Log(s-1))$, where $\Log$ is the principal determination of the logarithm, defined on $\C\backslash ]-\infty,0]$.

We begin by introducing some classical objects from the analytic theory of numbers.

\subsection{Notations}
We denote by $P$ the set of prime numbers. For $a$ and $m$ natural integers, we denote by
$P_a(m)=P\cap(a+m\Z)$ (the prime numbers congruent to $a$ modulo $m$).

For $m$ a non-zero natural number, we note $\Dir(m)$ for the set of characters modulo $m$: 
these are the multiplicative morphisms from $(\Z/m\Z)^{\times}$ to $\C$.
For $\chi\in\Dir(m)$ and $n$ a natural integer, we also denote by $\chi(n)$ the image of the class of $n$ in $\Z/m\Z$, 
the integers that are not prime to $m$ being sent to $0$. The function thus defined is multiplicative.

For real $a$, we denote by $H_a$ the set of complex numbers with a real part strictly greater than $a$.

We denote by $\Log$ the principal determination of the logarithm.
We recall that $(H_0,\times)$ is a group, and that $\forall a,b\in H_0$, $\Log(ab)=\Log(a)+\Log(b)$.
For $\alpha\in\R^{\*}$ and $z\in H_0$, $z^{\alpha}$ denotes $\exp(\alpha \Log(z))$ and for $\alpha,\beta\in\R^{\*}$ and
$z\in H_0$, we have $\exp\Log z=z$ and $(z^{\alpha})^{\beta}=z^{\alpha\beta}$.

In what follows, we will denote $E=I_m(A)$, $F=E\cap Q$, then

$$\forall s\in H_1\quad L(s)=\sum_{n=1}^{+\infty} \frac{\1_{E}(n)}{n^s}\text{ and }\tilde{L}(s)=\sum_{n=1}^{+\infty} \frac{\1_{F}(n)}{n^s}.$$

\subsection{Calculation of  \texorpdfstring{$L$ and $\tilde{L}$}{L}}

To calculate $L$ and $\tilde{L}$, we introduce a probabilistic formalism, which, we believe, 
sheds a rather pleasant light on certain properties, 
which could also be deduced from the factorization formula for $L$ functions associated with multiplicative functions.

For $s>1$, we consider a probability measure on $\N^*$, called the Zeta distribution with parameter $s$, which we denote $\mu_s$, and which is defined by $\mu_s(k)=\zeta(s)^{-1} k^{-s}$.

We thus have $\mu_s(E)=\frac{L(s)}{\zeta(s)}$.

If a random variable $X$ follows the $\mu_s$ law, it is easy to see that for any non-zero natural integer $n$, $\P(n | X)=n^{-s}$. If $\nu_p$ denotes the $p$-adic valuation, we have
\begin{align*}\P(\nu_{p_1}(X)\ge a_1,\dots,\nu_{p_r}(X)\ge a_r)&=\P(\prod_{i=1}^r p_i^{a_i}| X)\\
&= (\prod_i p_i^{a_i})^{-s}=\prod_i \P(\nu_{p_i}(X)\ge a_i),
\end{align*} 
Thus, the variables $\nu_p(X)$ are independent variables with shifted geometric laws, with $1+\nu_p(X)\sim\mathcal{G}(1-p^{-s})$, in other words, for $k\in\N$, we have \[\P(\nu_p(X)=k)=(1-p^{-s})p^{-sk}.\]

\newcommand{\pautorise}{\substack{p \text{ prime}\\ p\not\in A+m\Z}}
\newcommand{\pinterdit}{\substack{p \text{ prime}\\ p\in A+m\Z}}
We then observe that $\{X\in E\}=\miniop{}{\cap}{\pinterdit}\{\nu_p(X)=0\}$ and
$$\{X\in F\}=\miniop{}{\cap}{\pinterdit}\{\nu_p(X)=0\} \cap \miniop{}{\cap}{\pautorise}\{\nu_p(X)\le 1\}.$$ Since the $\nu_p(X)$'s are independent variables, we get

$$\mu_s(F)=\prod_{\pinterdit}(1-p^{-s}) \prod_{\pautorise}(1-p^{-2s})\text{ and }\mu_s(E)=\prod_{\pinterdit}(1-p^{-s}).$$

Similarly, the associated Dirichlet series are related by the identity

$$\tilde{L}(s)={L}(s)\prod_{p\in P\backslash I_m(A)}(1-p^{-2s}) .$$

In the sequel, we will still denote $\mu_s(E)$ for the quotient $\frac{L(s)} {\zeta(s)}$ for $s\in H_1$.
By the analytical extension principle, the identities that have been demonstrated for $s\in ]1,+\infty[$ still hold on the whole set $H_1$.

\subsection{Proof}

Our strategy is to apply Delange's Theorem a with $b(s)=0$.

Since $s\mapsto \prod_{p\in P\backslash I_m(A)}(1-p^{-2s})$ defines a holomorphic function that does not vanish on $H_{1/2}$, working with $L$ or $\tilde{L}$ (i.e. with $E$ or $F$) is equivalent.

We can further decompose

\[\mu_s(E)=\prod_{\substack{p\in P\cap (A+m\Z)\\p|m}}  (1-p^{-s}) \prod_{\substack{a\in A\\a\wedge m=1}} \prod_{p\in P_a(m)}(1-p^{-s})\]

The first product $\displaystyle\prod_{\substack{p\in P\cap (A+m\Z)\\p|m}}  (1-p^{-s})$ has a finite number of terms and defines a holomorphic function on $H_0$. The most significant part is given by the products of the form $\displaystyle \prod_{p\in P_a(m)} (1-p^{-s})$.

The study of these products reproduces the line of proofs for Dirichlet's arithmetic progression theorem.

We begin with a few classical lemmas of complex analysis.

\begin{lemma}
\label{basic}
There exists a holomorphic function $h_0$ on $H_0$ with $h_0(1)=0$ such that
$$\forall s\in H_1\quad -\Log \zeta(s)=\Log(s-1)+h_0(s).$$
\end{lemma}
\begin{proof}
It is well known that there exists a holomorphic function $\phi$ on $H_0$ with
on $H_1$ the identity $\zeta(s)=\frac1{s-1}+\phi(s)$. The function
$\psi:s\mapsto 1+(s-1)\phi(s)$ is of course holomorphic and is equal to $1$ at $1$.
On the half-plane $H_1$, the function $\zeta$ takes values in $H_0$.
From this we deduce
the identity $-\Log \zeta(s)=\Log(s-1)+\Log\psi(s)$.
The identity $\zeta(s)=\frac{\psi(s)}{s-1}$ extends to $H_0\backslash\{1\}$;
sine $\zeta$ does not vanish on $H_0$, neither does $\psi$, which makes it possible to define
a logarithm of $\psi$ on $H_0$ which extends $\Log \psi$; this is the desired function for $h_0$.
\end{proof}

\begin{remark}\label{adeuxballes}
As a consequence of this Lemma, we see that in the hypotheses of Delange's Theorem a, we can replace
$(s-1)^{-\rho}$ by $\zeta(s)^{\rho}$, because of the identity
\[\zeta(s)^{\rho}=\exp(-\rho \Log(s-1)-\rho h_0(s))=(s-1)^{-\rho}\exp(-\rho h_0(s)).\]
\end{remark}

\begin{lemma}
\label{undemi}
Let $(a_n)$ be a family of complex numbers with modulus less than or equal to $1$.
We set $F(s)=\prod_{p\in P}(1-\frac{a_p}{p^s})$.

For $s\in H_1$, the product $F(s)$ is absolutely convergent. There exists a holomorphic function $h$ on $H_{1/2}$ such that

$$\forall s\in H_1\quad -\Log F(z)=\sum_{p\in P} \frac{a_p}{p^s}+h(s).$$ 
\end{lemma}
\begin{proof} The proof is standard; it is given in abbreviated form: for $\Re s>1$, we write $-\Log(1-\frac{a_p}{p^s})=\frac{a_p}{p^s}+r(\frac{a_p}{p^s})$, with $r(z)=\sum_{n\ge 2} \frac{z^n}n$. For $\Re s\ge s_0>1/2$, we have $|r(\frac{a_p}{p^s})|\le r(p^{-s_0})$; the general term series
$r(p^{-s_0})$ converges, which ensures the holomorphy of $h(s)=\sum_{p} r(\frac{a_p}{p^s})$ on $H_{1/2}$.
\end{proof}
\begin{coro} 
The function defined on $H_1$ by $s\mapsto \sum_{p\in P}\frac{1}{p^s}$ extends to a holomorphic function on a neighborhood of $\overline{H_1}$ with $1$ removed. There exists a holomorphic function $h$ on a neighborhood of $\overline{H_1}$ with for $s\in H_1$: 
$$\sum_{p\in P} \frac{1}{p^s}=\Log\frac1{s-1}+h(s).$$
\end{coro}
\begin{proof}
According to Lemma~\ref{undemi} and the writing of $\zeta$ in the form of an Euler product, there exists a holomorphic function $h$ on $H_{1/2}$ such that for all $s\in H_{1}$ $$\Log\zeta(s)=\sum_{p\in P}\frac1{p^s}+h(s).$$
Thus, to extend $\sum_{p\in P}\frac1{p^s}$, it suffices to note that the $\zeta$ function, which is holomorphic on $H_0\backslash\{1\}$, does not vanish on
$H_1$ (where the product is absolutely convergent), nor on the line $\Re s=1$ deprived of 1 (see for example Colmez, theorem A.4.3 page 317). This allows to take the logarithm to extend to the line $\Re s=1$ minus $1$.
We conclude with lemma~\ref{basic}. 
\end{proof}

Dirichlet characters come into play.

\begin{coro}
\label{lescarac}
Let $\chi\in\Dir(m)$. We note $L(\chi,s)=\sum_{n\ge 1}\frac{\chi(n)}{n^s}$.
Then,
\begin{enumerate}
\item there exists a holomorphic function $h$ on $H_{1/2}$ such that
$$\forall s\in H_1\quad \Log L(\chi,s)=\sum_{p\in P}\frac{\chi(p)}{p^s}+h(s).$$
\item For $\chi\ne \chi_0$, the function
$s\mapsto \sum_{p\in P}\frac{\chi(p)}{p^s}$ extends to a holomorphic function on a neighborhood of $\overline{H_1}$.
\item For $\chi=\chi_0$, the function
$s\mapsto \sum_{p\in P}\frac{\chi_0(p)}{p^s}$ extends to a holomorphic function on a neighborhood of $\overline{H_1}$ with $1$ removed. There exists a holomorphic function $h$ on a neighborhood of $\overline{H_1}$ with for $s\in H_1$: 
$$\sum_{p\in P} \frac{\chi_0(p)}{p^s}=\Log\frac1{s-1}+h(s).$$
\end{enumerate} 
\end{coro}
\begin{proof}
With lemma~\ref{undemi}, the first point is an immediate consequence of writing in the form of an Euler product: $L(\chi,s)=\prod_{p\in P}(1-\frac{\chi(p)}{p^s})^{-1}$. As before, to extend $\sum_{p\in P}\frac{\chi(p)}{p^s}$, it is sufficient to note that the function $L(\chi,s)$, which is holomorphic on $H_0$, does not cancel out on
$H_1$ (where the product is absolutely convergent), nor on the line $\Re s=1$ (see for example Colmez, exercise A.4.4 page 318), which allows us to take the logarithm to extend on the line $\Re s=1$ deprived of $1$.
We conclude with lemma~\ref{basic}.
The last point follows from the fact that the series
$\sum_{p\in P} \frac{\chi_0(p)}{p^s}$ and $\sum_{p\in P} \frac{1}{p^s}$ only differ by a finite number of terms, that are holomorphic on $H_0$.
\end{proof}

\begin{lemma}
Let $a,m$ be non-zero natural number that are coprimes. Then there exists a holomorphic function $h_a$ at every point of $\overline{H_1}\backslash\{1\}$ such that for every $s\in H_1$
$$-\Log {\prod_{p\in P_a(m)}(1-p^{-s})}=\frac1{\phi(m)}\Log\frac1{s-1}+h_a(s).$$
\label{inter}
\end{lemma}
\begin{proof}
Given lemma~\ref{undemi}, it is sufficient to study
$$\sum_{p\in P_a(m)}\frac1{p^s}=\sum_{p\in P}\1_{a+m\Z}(p)\frac1{p^s}.$$
Now the theory of characters, applied in the commutative group $(\Z/m\Z)^{\times}$, tells us that for $a$ and $x$ in $(\Z/m\Z)^{\times}$, one has
$$\frac1{\phi(m)}\sum_{\chi\in\Dir(m)}\chi(x)=\delta_0(x)\text{ then }\frac1{\phi(m)}\sum_{\chi\in\Dir(m)}\chi(a^{-1}x)=\delta_0(a^{-1}x)=\delta_a(x)$$
By switching to integers, we obtain
$$\1_{a+m\Z}(p)=\frac1{\phi(m)}\sum_{\chi\in\Dir(m)} \overline{\chi(a)}\chi(p)$$
and
$$\sum_{p\in P_a(m)}\frac1{p^s}=\frac1{\phi(m)}\sum_{\chi\in\Dir(m)} \overline{\chi(a)} \sum_{p\in P} \frac{\chi(p)}{p^s}.$$
Then, the lemma immediately follows from Corollary~\ref{lescarac}.
\end{proof}
We can now prove Theorem~\ref{montheo}.
\begin{proof}

Let $s\in H_1$.
\( M(s)=\zeta(s) \prod_{\substack{a\in A\\a\wedge m=1}} \prod_{\substack{p\in P\\p\equiv a}}(1-p^{-s}) \) leads to
\[ \forall s\in H_1\quad \Log M(s)=\Log \zeta(s)+\sum_{\substack{a\in A\\a\wedge m=1}} \Log \prod_{p\in P_a(m)} (1-p^{-s}). \]
By putting together lemma~\ref{basic} and lemma~\ref{inter},
we obtain for $s\in H_1$
$$\Log M(s)=(-1+\frac{\ell}{\phi(m)})\Log (s-1)-\left(h_0(s)+\sum_{\substack{a\in A\\a\wedge m=1}} h_a(s)\right),\text{ then}$$

\[M(s)=(s-1)^{-1+\ell/\phi(m)}e^{-h(s)}\text{ where } h(s)=h_0(s)+\sum_{\substack{a\in A\\a\wedge m=1}} h_a(s),\text{ and}\] 
\[L(s)=\zeta(s)\mu_s(E)=\prod_{\substack{{p\in P\cap (A+m\Z)}\\{p|m}}} (1-p^{-s}) M(s)=(s-1)^{-1+\ell/\phi(m)}e^{-h(s)}\prod_{\substack{{p\in P\cap (A+m\Z)}\\{p|m}}} (1-p^{-s}). \] 

For $a$ between $1$ and $m$, we finally set
\[c_a=\begin{cases}
e^{-h_a(1)} & \text{ if }a\wedge m=1\\
1-a^{-1} &\text{ if }a\wedge m\ne 1\text{ and }a\in P\\
1 &\text{ if }a\wedge m\ne 1 \text{ and }a\not\in P\\
\end{cases}\]

Note that the $(c_a)$ do not depend on the set $A$. It can be seen that if $p\in P$, the conditions $p|m$ and $p\wedge m\ne 1$ are equivalent. It can then be verified that
\[e^{-h(1)}\prod_{\substack{{p\in P\cap (A+m\Z)}\\{p|m}}} (1-p^{-1})=\prod_{a\in A} c_a. \] 

When $\ell<\phi(m)$, we conclude with Delange's Theorem.

Suppose $\ell=\phi(m)$, and let $a\in\{1,\dots,m\}\backslash A$. Since $\ell=\phi(m)$, we have $a\wedge m\ne 1$.
Now let $p$ be a prime number congruent to $a$ modulo $m$.
$p$ is congruent to $a$ modulo $m\wedge a$, and $a$ is congruent to $0$ modulo $m\wedge a$,
so $p$ is congruent to $0$ modulo $m\wedge a$, and therefore $p=m\wedge a$. But we know that $p$ is congruent to $a$ modulo $m$, so $m\wedge a$ is congruent to $a$ modulo $m$, which gives us $a=m\wedge a=p$.
Finally, $p$ is a prime number that is a divisor of $m$ and a representative of an authorized congruence class.
Thus, if we denote by $(p_1,\dots,p_k)$ the set of such prime numbers, $I_m(A)$ is formed of integers that can be written in the form
$\prod_{i=1}^k p_i^{x_i}$, with $x_i$ a natural integer.
$I_m(A)\cap\{1,\dots, n\}$ therefore has the same cardinality as the set of points with integer coordinates in the simplex determinated by the  equations
\[\begin{cases}\forall i\in \{1,\dots,k\} \quad x_i\ge 0\\ \sum_{i=1}^k x_i\log(p_i)\le \log n\end{cases}\]
In other words, $|I_m(A)\cap\{1,\dots, n\}|= | (\log n)V \cap \Zd|$, with
$$V=\left\{ x\in [0,+\infty[^k ; \sum_{i=1}^k x_i\log(p_i)\le 1\right\},$$
which leads to $|I_m(A)\cap\{1,\dots, n\}|\sim (\log n)^k \lambda(V)= (\log n)^k \frac1{k!}\prod_{i=1}^k \frac1{\log p_i}$. 

\end{proof}

We can move on to the proof of Theorem~\ref{theocubes}.

\begin{proof}Since $W=Q\cap I_3(\{1\})$, it follows from the second point of Theorem~\ref{montheo} . We have $A=\{1\}$, $m=3$, $\ell=1<2=\phi(m)$.
\end{proof} 

\section{Calculation of the constant \texorpdfstring{$C$}{C}}\label{IV}

We introduce the primitive Dirichlet character modulo~3, denoted by $\chi_1$: it is defined by

\[ \chi_1(n) = \begin{cases} 1 & \text{if } n \equiv 1 \pmod{3}, \\ -1 & \text{if } n \equiv 2 \pmod{3}, \\ 0 & \text{if } n \equiv 0 \pmod{3}. \end{cases} \] This character allows us to separate the prime numbers according to their congruence modulo~3.

We know that if a function $\phi$ is strictly multiplicative with a modulus bounded by 1, we have
\begin{align*} \sum_{n\ge 1}\frac{\phi(n)}{n^s}&=\prod_{p\in P} \left(1-\frac{\phi(p)}{p^s}\right)^{-1}.
\end{align*}
In particular, by separating according to the congruences modulo $3$, we have
\begin{align*} \sum_{n\ge 1}\frac{\phi(n)}{n^s}
&=\left(1-\frac{\phi(3)}{3^s}\right)^{-1}\prod_{p\in P_1(3)} \left(1-\frac{\phi(p)}{p^s}\right)^{-1}\prod_{p\in P_2(3)} \left(1-\frac{\phi(p)}{p^s}\right)^{-1}.
\end{align*}
Applying this to the character $\chi_1$, we obtain for the function $L_{\chi_1}(s)=\sum_{n\ge 1} \frac{\chi_1(s)}{n^s}$: 
\[ L_{\chi_1}(s) = \prod_{p\in P_1(3)} \left(1 - \frac{1}{p^s}\right)^{-1} \cdot \prod_{p\in P_2(3)} \left(1 + \frac{1}{p^{s}}\right)^{-1}, \]

then, multiplying by the expansion of $\zeta(s)$, we successively get

\begin{align*}\zeta(s)L_{\chi_1}(s)&=\left(1-\frac{1}{3^s}\right)^{-1}\prod_{p\in P_1(3)} \left(1-\frac{1}{p^s}\right)^{-2}\prod_{p\in P_2(3)} \left(1-\frac{1}{p^{2s}}\right)^{-1}\\
\mu_s\left(E\right)^2&=\zeta\left(s\right)^{-1}L_{\chi_1}\left(s\right)^{-1}\left(1-\frac{1}{3^s}\right)^{-1}\prod_{p\in P_2(3)} \left(1-\frac{1}{p^{2s}}\right)^{-1},\\
\mu_s(E)&=\zeta\left(s\right)^{-1/2}L_{\chi_1}\left(s\right)^{-1/2}\left(1-\frac{1}{3^s}\right)^{-1/2}\prod_{p\in P_2(3)} \left(1-\frac{1}{p^{2s}}\right)^{-1/2},\end{align*}

and finally
\[L(s)=\zeta(s)^{1/2}L_{\chi_1}(s)^{-1/2}\left(1-\frac{1}{3^s}\right)^{-1/2}\prod_{p\in P_2(3)} \left(1-\frac{1}{p^{2s}}\right)^{-1/2}.\]
and \begin{align*}\tilde{L}(s)&=\prod_{p\in P\backslash I_3(\{1\})}(1-p^{-2s}) L(s)\\
&=\zeta(s)^{1/2}L_{\chi_1}(s)^{-1/2}\left(1-\frac{1}{3^s}\right)^{-1/2}\left(1-9^{-s}\right)\prod_{p\in P_2(3)} \left(1-\frac{1}{p^{2s}}\right)^{1/2}.
\end{align*}

Applying Delange's Theorem and Remark~\ref{adeuxballes}, we find
\[c=\frac1{\sqrt{\pi}} L_{\chi_1}(1)^{-1/2}(2/3)^{-1/2}\frac{8}{9}\sqrt{p_2}, \text{ with }p_2=\prod_{p\in P_2(3)} \left(1-\frac{1}{p^{2}}\right).\]
The value of $L_{\chi_1}(1)$ can be easily calculated: we have
\begin{align*}L_{\chi_1}(1)&=\sum_{n\ge 0}\left(\frac1{3n+1}-\frac1{3n+2}\right)=\sum_{n\ge 0}\int_0^1 (x^{3n}-x^{3n+1})\ dx=\int_0^1\frac{dx}{x^2+x+1}\\
&= \left[
\frac{2}{\sqrt 3} \arctan\left(\frac{1 + 2 x}{\sqrt{3}}\right)\right]_0^1
=\frac{2}{\sqrt 3}\left(\frac{\pi}3-\frac\pi{6}\right)=\frac{\pi}{3\sqrt{3}}
\end{align*}
and finally
$C=\frac{4}{\pi}3^{-3/4}\sqrt{2p_2}\approx 0.664$.

\def\refname{References}
\bibliographystyle{plain}
\bibliography{cubes-delange-eng}

\end{document}